\documentclass[11pt]{article}
\usepackage[latin1]{inputenc}
\usepackage{amsmath,amsthm,amssymb}
\usepackage{amsfonts}
\usepackage{amsmath,amsthm,amssymb,amscd}
\usepackage{latexsym}
\usepackage{color}
\usepackage{graphicx}
\usepackage{mathrsfs}
\usepackage{hyperref,xcolor}
\hypersetup{pdfborder={0 0 0}, colorlinks,
}
\makeatletter \@addtoreset{equation}{section} \makeatother

\newtheorem{theorem}{Theorem}[section]

\newtheorem{proposition}{Proposition}[section]
\newtheorem{lemma}{Lemma}[section]
\newtheorem{remark}{Remark}[section]

\begin{document}
\title{\sc Normalized ground state solutions for the fractional Sobolev critical NLSE  with an extra mass supercritical nonlinearity}
\author{{\sc Jiabin Zuo$^{\small\mbox{1}}$, Yuyou Zhong$^{\small\mbox{1}}$, Du\v{s}an D. Repov\v{s}$^{\small\mbox{2,3,4}}$\thanks{{ {{\it E-mail addresses:} zuojiabin88@163.com (J. Zuo), zhongyy@e.gzhu.edu.cn (Y. Zhong), dusan.repovs@guest.arnes.si (D.D. Repov\v{s}).
} }}
\thanks{{{Corresponding author: Du\v{s}an D. Repov\v{s},  dusan.repovs@guest.arnes.si}}}}\\
$^{\small\mbox{1}}${\small School of Mathematics and Information Science, Guangzhou University, 510006 Guangzhou, P.R. China}\\
$^{\small\mbox{2}}${\small Faculty of Education, University of Ljubljana,  1000 Ljubljana, Slovenia}\\
$^{\small\mbox{3}}${\small Faculty of Mathematics and Physics, University of Ljubljana, 1000 Ljubljana, Slovenia}\\
$^{\small\mbox{4}}${\small Institute of Mathematics, Physics and Mechanics, 1000 Ljubljana,  Slovenia}}
\date{}
\maketitle
\begin{abstract}
This paper is concerned with
 existence of  normalized ground state solutions for the mass supercritical  fractional nonlinear Schr\"{o}dinger equation involving a critical growth in the fractional Sobolev sense. The compactness of Palais-Smale sequences is obtained
 by a special technique, which borrows from the ideas of Soave
  (J. Funct. Anal. 279 (6) (2020), art. 1086102020). 
  This paper represents an extension of previously known results, both in the local and the nonlocal cases.

 \medskip

\medskip

\end{abstract}

%***********************************************************************************************************************************
\section{Introduction}\label{sec1}
%***********************************************************************************************************************************

~~This paper is devoted to the following fractional Sobolev critical nonlinear Schr\"{o}dinger equation (NLSE)
 in $\mathbb{R}^{N}$ ($N\geq2$)
\begin{equation}\label{p}
\left\{\begin{array}{lll}
(-\Delta)^{s} u=\mu u+|u|^{2^{*}_{s}-2}u+\eta|u|^{p-2}u,  \\
\|u\|^{2}_{L^{2}}=m^{2},
\end{array}\right.
\end{equation}
where $s\in(0,1)$,  $\mu\in\mathbb{R}$ is an unknown real number
 which will appear as a Lagrange multiplier, $2^{*}_{s}$ is the fractional Sobolev critical exponent, $\eta>0$, $p\in(2+\frac{4s}{N},2^{*}_{s})$, $m>0$ is a finite parameter, and $(-\Delta)^{s}$ is the fractional Laplace operator defined by
$$(-\Delta)^{s}u(x)=C(N,s)\lim\limits_{\varepsilon\rightarrow0^{+}}\int_{\mathbb{R}^{N}\backslash
B_{\varepsilon}(x)}\frac{u(x)-u(y)}{|x-y|^{N+2s}}dy,$$
where $C(N, s)$ is the dimensional constant
which
 depends on $N$ and $s$
 (for more details we refer
the interested reader to Di Nezza et al. \cite{edz}).

The fractional Schr\"{o}dinger equation originated from Laskin's work \cite{nl}, and in recent years, the interest in its study has grown
considerably. It is worthwhile and
very  interesting to look for normalized solutions to such equations
which have a prescribed $L^2$-norm, since
they represent the physical view of the
conservation of mass.

Existence and properties of normalized solutions for certain problems strongly depend
on the behavior of combined nonlinearity $g(u)=|u|^{q-2}u+\eta |u|^{p-2}u,$ where $2<p<q<2^{*}_{s}$.
 Zhang et al.
\cite{HZ}
 investigated a class of Sobolev subcritical fractional NLSEs where
  the parameters $p$ and $q$ are in different order,
and they
obtained several
interesting
 results concerning
 existence of normalized solutions.  In particular, when the degree of nonlinearity $g(u)$ exceeds the mass critical index $2+4s/N$, the functional
 turns out to be
  unbounded from below, which makes it
 impossible to  adopt the direct variational method.

 Secchi and Appolloni
 \cite{LA}
 studied existence and multiplicity of ground state normalized solutions for the fractional mass supercritical NLSEs by using the min-max theory under more general assumptions, but in the Sobolev subcritical sense. Under certain conditions on the potential, Peng et al.
 \cite{SP}
 showed that
 NLSE has at least one normalized solution, with the help of a new min-max argument and the splitting lemma for nonlocal version also in the case when the mass is supercritical and Sobolev subcritical.

However, to the best of our knowledge, there are very
few papers
 on normalized solutions of  fractional NLSEs. Moreover, they only consider the Sobolev subcritical case. Therefore it is natural to inquire what difficulties
  appear
  when
   a fractional Sobolev critical nonlinearity is considered.
For example,
 Zhen and Zhang \cite{MZ} investigated a critical fractional NLSE with the $L^2$-supercritical perturbation, but their coefficient of perturbation
was not allowed to be large enough, while in the present paper, $\eta$ can be large
enough
 since $\eta\in [\eta^{*}, +\infty)$. Very recently, almost at the same time as our study, Li and Zou \cite{lq} considered the same problem using the concentration compactness principle for overcoming the lack of compactness. With the above methods and techniques, Zuo and  R\v{a}dulescu \cite{zuo1} investigated the existence and nonexistence of normalized solutions for a class of fractional mass supercritical nonlinear Schr\"{o}dinger coupled systems with Sobolev
critical nonlinearities.
In response to this difficulty, we consider a technical analysis method combined with the Br\'{e}zis-Lieb lemma, which comes from the ideas of Soave~\cite{NS}.

Of course, in
 the case when $s\rightarrow1$, the fractional Laplacian $(-\Delta)^{s}$ reduces to the
  classical Laplace operator $-\Delta$, the literature on the relevant problem  \eqref{p} is very large.
   Here we shall only mention some  key papers which are relevant to our study. Br\'{e}zis and Nirenberg  \cite{hb} presented a
    pioneering work. Later, many made
important progress in this field. For $L^2$-supercritical perturbation $\eta |u|^{p-2}u$,  Soave \cite{NS} made the
 first contribution concerning existence  of normalized solutions for NLSEs in the Sobolev critical case. Next,
   Alves et al. \cite{coa}  obtained a similar result for this kind of NLSEs
     when
      dimension $N\geq5$,
     and
      $\eta$
       large enough.
 In particular, under weaker, more general conditions,  Jeanjean and Lu \cite{LJSS} proved existence  of ground
states and asymptotic behavior of ground state energy with mass change. They also
obtained infinitely
many radial solutions when $N\geq2$, and
established existence  and multiplicity of nonradial
sign-changing solutions for every $N\geq4$.

Inspired by the work mentioned above, we consider in this paper the problem of
existence  of ground state normalized solutions of fractional Sobolev critical NLSEs  with a mass supercritical nonlinearity.  The compactness can be restored combining some of the main ideas of Br\'{e}zis and Nirenberg  \cite{hb} and Jeanjean \cite{lj1}.
In order to introduce the main result of this paper, we first define fractional Sobolev space:
$$H^{s}(\mathbb{R}^{N})=\left\{u\in L^{2}(\mathbb{R}^{N})~\big|~ [u]^{2}_{H^{s}}=\int\int_{\mathbb{R}^{2N}}\frac{|u(x)-u(y)|^{2}}{|x-y|^{N+2s}}dxdy<\infty\right\},$$
in which
the norm is defined by $$\|u\|=(\|u\|^{2}_{L^{2}}+[u]^{2}_{H^{s}})^{\frac{1}{2}}.$$
For convenience, we shall simply denote the norm of Lebesgue space $L^p(\mathbb{R}^{N})$ by $\|u\|_{p}$ for $p\in[1,\infty)$.
A standard method for investigating problem \eqref{p} is to find critical points of the following energy functional
$$I_{\eta}(u)=\frac{1}{2}[u]^{2}_{H^{s}}-\frac{1}{2_{s}^{*}}\|u\|_{2_{s}^{*}}^{2_{s}^{*}}-\frac{\eta}{p}\|u\|_{p}^{p},$$
constrained to the
 set $$S(m)=\left\{u\in H^{s}(\mathbb{R}^{N})\big|\|u\|^{2}_{2}=m^{2}\right\}.$$ Obviously, $I_{\eta}$ is of class $C^{1}$ in $H^{s}(\mathbb{R}^{N})$.

We can now state our main result.
\begin{theorem}\label{T1.1}
Assume that $N\geq2$
and  $p\in(2+\frac{4s}{N},2^{*}_{s})$.  Then  for every $m>0,$ there exists $\eta^{*}=\eta^{*}(m)>0$ such that for every $\eta\geq\eta^{*}$, problem \eqref{p} admits a radial normalized solution $\widetilde{u},$ whose associated Lagrange multiplier $\mu$ is negative.
\end{theorem}
\begin{remark}\label{remark}
 Our conclusion can be regarded as an extension of  Alves et al. \cite[ Theorem 1.1]{coa}.
\end{remark}
\begin{remark}\label{remark}
According to Theorem 1.1 and Zhen and Zhang \cite[ Theorem 1.3:(1)-(2)]{MZ}, we know that existence  of  radial normalized ground state is possible when $\eta$ is small enough or large enough, however it remains an open problem for the rest of the range of $\eta$.
\end{remark}
\begin{remark}\label{remark}
If $N=4s$ or $\frac{p}{p-1}2s<N<4s$ or $2s<N<\frac{p}{p-1}2s$, we can get our conclusion without any restriction on $\eta$, see Zhen and Zhang \cite[ Theorem 1.3(3)]{MZ}. If we further assume that $N^2>8s^{2},$
then we can also arrive at a similar conclusion, but in this case we consider the mass $m$ as the parameter instead of $\eta$, see Zhang and Han \cite[Theorem 1.3]{ph}.
\end{remark}

\noindent The paper is organized as follows. Section \ref{sec2} contains the proofs of  some important lemmas, which play a key role in the proof of the compactness condition.  In
Section \ref{sec3}, we prove the strong convergence  of the
Palais-Smale sequence at  some level set, using a special technique. In Section \ref{sec4}, we prove Theorem \ref{T1.1}.

\section{Main lemmas}\label{sec2}%2
%\hskip \parindent

Although the study of normalized solutions is convenient for applications, it also presents some difficulties. For example, the Nehari manifold method cannot be used because the constant $\mu$ is unknown. This also makes it difficult to verify the boundedness of Palais-Smale sequences when employing some common methods.

To this end, following Soave
 \cite{NS}, let
\begin{equation}\label{add1}
\zeta_{p}=(Np-2N)/2ps,
\
\hbox{for every}
\
p\in (2,2_{s}^{*}],
\end{equation}
it is easy to see that $\zeta_{p}\in(0,1]$ and define the Pohozaev manifold
$$\mathcal{P}_{\eta,m}=\{u\in S(m)|P_{\eta}(u)=0\},
\
\hbox{where}
\
P_{\eta}(u)=[u]^{2}_{H^{s}}-\|u\|_{2_{s}^{*}}^{2_{s}^{*}}-\eta\zeta_{p}\|u\|_{p}^{p},$$
where the definition of $\zeta_{p}$
is related
 to \eqref{FNG}.
It is well-known
 that any critical point of
$I_{\eta}|_{S(m)}$ stays in $\mathcal{P}_{\eta,m}$, as a consequence of Zhen and Zhang \cite[Proposition 2.1 and Remark 2.1]{MZ}.

In order to get the mountain pass geometry, we
are
 going to make the scaling transformation.  For $u\in S(m)$ and $\xi\in \mathbb{R}$, we let
$$(\xi\star u)(x)=e^{\frac{N\xi}{2}}u(e^{\xi}x)=v(x),~\text{for~a.e.~} x\in\mathbb{R}^{N},$$ which is based on a very interesting idea from Jeanjean \cite{lj1}.
A careful analysis shows that the transformed functional $\widetilde{I}_{\eta}=I_{\eta}(\xi\star u)$ has the same mountain pass geometry and mountain pass level as the original functional $I_{\eta}(u)$.

For reader's convenience, we give the proof of the following  lemma, which can also be found in Li and Zou
\cite{lq}.

\begin{lemma}\label{Lemma1}
Assume that $u\in S(m)$ is arbitrary but fixed. Then we have:

\noindent(1)~~$[\xi\star u]^{2}_{H^{s}}\rightarrow0$~\text{and}~ $I_{\eta}(\xi\star u)\rightarrow0,$ ~\text{as}~ $\xi\rightarrow-\infty;$~\\
\noindent(2)~~$[\xi\star u]^{2}_{H^{s}}\rightarrow+\infty$~\text{and}~ $I_{\eta}(\xi\star u)\rightarrow-\infty,$ ~\text{as}~ $\xi\rightarrow+\infty.$~

\end{lemma}

\begin{proof}
By a direct calculation, we get
\begin{equation}\label{j}
\begin{split}
[\xi\star u]^{2}_{H^{s}}&=e^{2\xi s}\int\int_{\mathbb{R}^{2N}}\frac{|u(x)-u(y)|^{2}}{|x-y|^{N+2s}}dxdy=e^{2\xi s}[u]^{2}_{H^{s}},\\ ~~~\|\xi\star u\|_{\beta}^{\beta}&=e^{\frac{(\beta-2)N\xi}{2}}\|u\|_{\beta}^{\beta}, \ \hbox{for every } \  \beta\geq2.
\end{split}
\end{equation}
On the basis of \eqref{j}, we have

$$[\xi\star u]^{2}_{H^{s}}\rightarrow0,~~\hbox{as}~~\xi\rightarrow-\infty,$$
 $$I_{\eta}(\xi\star u)=\frac{1}{2}[\xi\star u]^{2}_{H^{s}}-\frac{1}{2_{s}^{*}}\|\xi\star u\|_{2_{s}^{*}}^{2_{s}^{*}}-\frac{\eta}{p}\|\xi\star u\|_{p}^{p}\rightarrow0,~~\hbox{as}~~\xi\rightarrow-\infty,$$

thereby demonstrating (1).

On the other hand, it follows from \eqref{j} that $[\xi\star u]^{2}_{H^{s}}\rightarrow+\infty,$ ~\hbox{as}~ $\xi\rightarrow+\infty.$
Besides,
\begin{align*}
 I_{\eta}(\xi\star u)=\frac{1}{2}e^{2\xi s}[u]^{2}_{H^{s}}-\frac{1}{2_{s}^{*}}e^{\frac{(2_{s}^{*}-2)N\xi}{2}}\|u\|_{2_{s}^{*}}^{2_{s}^{*}}-\frac{\eta}{p}e^{\frac{(p-2)N\xi}{2}}\|u\|_{p}^{p}\rightarrow-\infty,~~\hbox{as}~~\xi\rightarrow+\infty,
\end{align*}
since $p\in(2+\frac{4s}{N},2^{*}_{s})$, which in turn, demonstrates
 (2). This completes the proof of Lemma \ref{Lemma1}.
\end{proof}

The following two inequalities (fractional Sobolev inequality \eqref{sobolev} and fractional Gagliardo-Nirenberg inequality \eqref{FNG}) play an important role in our proof of the main result in Section \ref{sec4}.

Thanks to Servadei and Valdinoci \cite{RS}, there exists a optimal fractional critical Sobolev constant $\mathcal{S}>0$ such that
\begin{equation}\label{sobolev}
\mathcal{S}\|u\|_{2_{s}^{*}}^{2}\leq [u]^{2}_{H^{s}}, \ \hbox{for every } \ u\in H^{s}(\mathbb{R}^{N}).
\end{equation}
Also, according to Frank et al. \cite{FRL}, there exists an optimal constant $C(N,p,s)$ such that for
every
 $p\in(2,2_{s}^{*})$,
\begin{equation}\label{FNG}
\|u\|_{p}^{p}\leq C^{p}(N,p,s)[u]^{p\zeta_{p}}_{H^{s}}\|u\|_{2}^{p(1-\zeta_{p})},~~\hbox{for every } \ u\in H^{s}(\mathbb{R}^{N}),
\end{equation}
where $\zeta_{p}$ is given by \eqref{add1}.

In the following lemma, we give a
 specific value of $\rho(m,\eta)$ and analyze the asymptotic behavior of $\rho(m,\eta)$ when $\eta$ is large  enough, which is more detailed than in    Luo and Zhang \cite[Lemma 5.2]{HZ}.

 Let $$S_{r}(m)=S(m)\bigcap H^{s}_{rad}(\mathbb{R}^{N})=\{u\in S(m): u(x)=u(|x|)\}.$$

\begin{lemma}\label{Lemma2}
There exists
a
 small enough
 $\rho(m,\eta)>0$  such that $$0<\inf\limits_{u\in X} I_{\eta}(u)\leq\sup\limits_{u\in X} I_{\eta}(u)<\inf\limits_{u\in Y} I_{\eta}(u),$$
with
$$X=\left\{u\in S_{r}(m), [u]^{2}_{H^{s}}\leq\rho(m,\eta)\right\},~ Y=\left\{u\in S_{r}(m), [u]^{2}_{H^{s}}=2\rho(m,\eta)\right\}.$$
Moreover, $\rho(m,\eta)\rightarrow0,$ as $\eta\rightarrow\infty$.
\end{lemma}
\begin{proof}
 In view of \eqref{sobolev} and \eqref{FNG}, we get
\begin{equation}\label{0}
\frac{1}{2_{s}^{*}}\|v\|_{2_{s}^{*}}^{2_{s}^{*}}+\frac{\eta}{p}\|v\|_{p}^{p}\leq \frac{1}{2_{s}^{*}\mathcal{S}^{\frac{2_{s}^{*}}{2}}}([v]^{2}_{H^{s}})^{\frac{2_{s}^{*}}{2}}+\frac{\eta C^{p}(N,p,s)}{p}([v]^{2}_{H^{s}})^{\frac{Np-2N}{4s}}m^{\frac{2sp-Np+2N}{2s}}.
\end{equation}
Then for every $u\in S_{r}(m)$, fixing $[u]^{2}_{H^{s}}\leq\rho(m,\eta)$ and $[v]^{2}_{H^{s}}=2\rho(m,\eta)$, where $\rho(m,\eta)$ is a positive number
which
 depends on $m$ and $\eta$, one has
\begin{align*}
\begin{split}
I_{\eta}(v)-I_{\eta}(u)&\geq \frac{1}{2}[v]^{2}_{H^{s}}-\frac{1}{2}[u]^{2}_{H^{s}}-\frac{1}{2_{s}^{*}}\|v\|_{2_{s}^{*}}^{2_{s}^{*}}-\frac{\eta}{p}\|v\|_{p}^{p}\\
&\geq \frac{1}{2}\rho(m,\eta)-\frac{2^{\frac{2_{s}^{*}}{2}}}{2_{s}^{*}\mathcal{S}^{\frac{2_{s}^{*}}{2}}}(\rho(m,\eta))^{\frac{2_{s}^{*}}{2}}-\frac{\eta C^{p}(N,p,s)}{p}2^{\frac{Np-2N}{4s}}(\rho(m,\eta))^{\frac{Np-2N}{4s}}m^{\frac{2sp-Np+2N}{2s}}.
\end{split}
\end{align*}
Thus, choosing
\begin{align}\label{2}
\rho(m,\eta)=\min\left\{\left(\frac{p}{8\eta C^{p}(N,p,s)2^{\frac{Np-2N}{4s}}m^{\frac{2sp-Np+2N}{2s}}}\right)^{\frac{4s}{Np-2N-4s}}, \left(\frac{2_{s}^{*}}{8}\right)^{\frac{N-2s}{2s}}\left(\frac{ \mathcal{S}}{2}\right)^{\frac{N}{2s}} \right\},
\end{align}
we can infer
that
$$\frac{1}{2}\rho(m,\eta)-\frac{2^{\frac{2_{s}^{*}}{2}}}{2_{s}^{*}\mathcal{S}^{\frac{2_{s}^{*}}{2}}}(\rho(m,\eta))^{\frac{2_{s}^{*}}{2}}-\frac{\eta C^{p}(N,p,s)}{p}2^{\frac{Np-2N}{4s}}(\rho(m,\eta))^{\frac{Np-2N}{4s}}m^{\frac{2sp-Np+2N}{2s}}>0.$$
Now, by \eqref{0} and the definition of $\rho(m,\eta)$ in \eqref{2}, we get
$$I_{\eta}(u)\geq\frac{1}{2}[u]^{2}_{H^{s}}-\frac{2^{\frac{2_{s}^{*}}{2}}}{2_{s}^{*}\mathcal{S}^{\frac{2_{s}^{*}}{2}}}[u]^{2_{s}^{*}}_{H^{s}}-\frac{\eta C^{p}(N,p,s)}{p}2^{\frac{Np-2N}{4s}}m^{\frac{2sp-Np+2N}{2s}}[u]^{\frac{Np-2N}{2s}}_{H^{s}}>0,$$
which means
that the
 inequality in Lemma \ref{Lemma2} holds. Finally, $\lim\limits_{\eta\rightarrow\infty}\rho(m,\eta)=0$ also follows  from \eqref{2}.
 This completes the proof of Lemma \ref{Lemma2}.
\end{proof}

Next,  fix $u_{0}\in S_{r}(m).$
It follows from Lemma \ref{Lemma1} and Lemma \ref{Lemma2} that there exist  numbers $\xi_{1}=\xi_{1}(m,\eta,u_0)<0$
and
 $\xi_2=\xi_2(m,\eta,u_0)>0$ such that the  functions $u_{1,\eta}=\xi_{1}\star u_0, u_{2,\eta}=\xi_{2}\star u_0$ satisfy
$$[u_{1,\eta}]^{2}_{H^{s}}<\frac{\rho(m,\eta)}{2}, [u_{2,\eta}]^{2}_{H^{s}}>2\rho(m,\eta),
\
I_{\eta}(u_{1,\eta})>0,
\
\hbox{and}
\quad
 I_{\eta}(u_{2,\eta})<0.$$

Now, similarly to the discussion in Jeanjean \cite{lj1} or  Luo and Zhang \cite[Proposition 5.3]{HZ}, we fix the minimax
$$E_{\eta}(m)=\inf\limits_{\psi\in \Gamma}\max\limits_{t\in(0,1]}I_{\eta}(\psi(t)),
\
\hbox{where}
\
\Gamma=\{\psi\in C([0,1], S_{r}(m)): [\psi(0)]^{2}_{H^{s}}<\rho(m,\eta)/2, ~ I_{\eta}(\psi(1))<0\}.$$

By virtue of Lemma \ref{Lemma2}, we know that
$$[\psi(1)]^{2}_{H^{s}}>\rho(m,\eta),
\
\hbox{for every}
\
\psi\in\Gamma.$$
Therefore there exists  $t_{0}\in(0,1)$ such that $$[\psi(t_{0})]^{2}_{H^{s}}=\rho(m,\eta)/2 ~~\text{and}~~\max\limits_{t\in[0,1]}I_{\eta}(\psi(t))\geq I_{\eta}(\psi(t_{0}))\geq\inf\limits_{u\in X} I_{\eta}(u)>0,$$
therefore $E_{\eta}(m)>0$.

 The following lemma is a key step to analyze the level value of mountain pass, so we present a more detailed calculation process (in comparison with Li and Zou \cite{lq}).

\begin{lemma}\label{Lemma3}
 $\lim\limits_{\eta\rightarrow\infty}E_{\eta}(m)=0$.
\end{lemma}
\begin{proof}
Fix $u_{0}\in S_{r}(m)$, and consider
 the path $\psi_{0}(t)=[(1-t)\xi_{1}+t\xi_{2}]\star u_{0}\in\Gamma$.
We have
$$E_{\eta}(m)\leq\max\limits_{t\in[0,1]}I_{\eta}(\psi_{0}(t))\leq\max\limits_{r\geq0}\left\{\frac{1}{2}r^{2}[u_{0}]^{2}_{H^{s}}-\frac{\eta}{p}r^{\frac{Np-2N}{2s}}\|u_{0}\|_{p}^{p}\right\}.$$
Thus, setting $C_{1}=[u_{0}]^{2}_{H^{s}}$ and $C_{2}=\|u_{0}\|_{p}^{p}$, we consider the maximum value of the following function
$$f(r)=\frac{1}{2}C_{1}r^{2}-\frac{\eta}{p}C_{2}r^{\frac{Np-2N}{2s}}, ~\text{for~any}~r\geq0.$$
Letting
$$f^{'}(r)=C_{1}r-\left(\frac{Np-2N}{2s}\right)\frac{\eta}{p}C_{2}r^{\frac{Np-2N-2s}{2s}}=0,$$
we get the maximum
of $f(r)$ at
$$r_{max}=\left(\frac{2spC_{1}}{(Np-2N)\eta C_{2}}\right)^{\frac{2s}{Np-2N-4s}}.$$
Hence,
%\begin{align*}
%\begin{split}
$$
\max\limits_{r\geq0}\left\{\frac{1}{2}r^{2}[u_{0}]^{2}_{H^{s}}-\frac{\eta}{p}r^{\frac{Np-2N}{2s}}\|u_{0}\|_{p}^{p}\right\}
$$
$$
=\frac{1}{2}\left(\frac{2spC_{1}}{(Np-2N)\eta C_{2}}\right)^{\frac{4s}{Np-2N-4s}}C_{1}-\frac{\eta}{p}\left(\frac{2spC_{1}}{(Np-2N)\eta C_{2}}\right)^{\frac{Np-2N}{Np-2N-4s}}C_{2}
$$
$$
\leq\frac{1}{2}\left(\frac{2spC_{1}}{(Np-2N)\eta C_{2}}\right)^{\frac{4s}{Np-2N-4s}}C_{1},
$$
%\end{split}
%\end{align*}
so there exists $C>0$, which is not dependent on $\eta>0$, such that
$$E_{\eta}(m)\leq C\left(\frac{1}{\eta}\right)^{\frac{4s}{Np-2N-4s}}\rightarrow0, ~\text{as}~ \eta\rightarrow\infty, $$
since $p>2+4s/N$. This completes the proof of Lemma \ref{Lemma3}.
\end{proof}

 Similarly to Luo and Zhang \cite[Propositions 5.3-5.4]{HZ}, for $\{\xi_{n}\}\subset \mathbb{R}$ we know that $I_{\eta}(u_{n})$ and $I_{\eta}(\xi_{n}\star u_{n})$ have the same mountain pass level value. Moreover, there is a certain relationship between their Palais-Smale sequences.

 \begin{lemma}\label{Lemma4}
Let $\{\xi_{n}\star u_{n}\}\subset S_{r}(m)$ be
  a Palais-Smale sequence for $I_{\eta}$ at  level $E_{\eta}(m),$ i.e.,
 $$I_{\eta}(\xi_{n}\star u_{n})\rightarrow E_{\eta}(m)>0
 \
 \hbox{and}
 \
I^{'}_{\eta}(\xi_{n}\star u_{n})\rightarrow 0,~\text{as} ~n\rightarrow\infty. ~Then ~\lim\limits_{n\rightarrow\infty}P_{\eta}(\xi_{n}\star u_{n})=0.$$
\end{lemma}
\begin{proof}

We first have
\begin{align*}
\begin{split}
I_{\eta}(\xi_{n}\star u_{n})&=\frac{1}{2}[\xi_{n}\star u_{n}]^{2}_{H^{s}}-\frac{1}{2_{s}^{*}}\|\xi_{n}\star u_{n}\|_{2_{s}^{*}}^{2_{s}^{*}}-\frac{\eta}{p}\|\xi_{n}\star u_{n}\|_{p}^{p}\\&= \frac{1}{2}e^{2\xi_{n} s}[u_{n}]^{2}_{H^{s}}-\frac{1}{2_{s}^{*}}e^{\frac{(2_{s}^{*}-2)N\xi_{n}}{2}}\|u_{n}\|_{2_{s}^{*}}^{2_{s}^{*}}-\frac{\eta}{p}e^{\frac{(p-2)N\xi_{n}}{2}}\|u_{n}\|_{p}^{p},
\end{split}
\end{align*}
and  $I_{\eta}(\xi_{n}\star u_{n})$ is $C^{1}$ with respect to $\xi_{n}$. Now, by taking the derivative
\begin{align*}
\frac{\partial}{\partial \xi_{n}}I_{\eta}(\xi_{n}\star u_{n})= 2se^{2\xi_{n} s}[u_{n}]^{2}_{H^{s}}-se^{\frac{(2_{s}^{*}-2)N\xi_{n}}{2}}\|u_{n}\|_{2_{s}^{*}}^{2_{s}^{*}}-s\eta\zeta_{p}e^{\frac{(p-2)N\xi_{n}}{2}}\|u_{n}\|_{p}^{p},
\end{align*}
and noticing  that
$$P_{\eta}(\xi_{n}\star u_{n})= 2e^{2\xi_{n} s}[u_{n}]^{2}_{H^{s}}-e^{\frac{(2_{s}^{*}-2)N\xi_{n}}{2}}\|u_{n}\|_{2_{s}^{*}}^{2_{s}^{*}}-\eta\zeta_{p}e^{\frac{(p-2)N\xi_{n}}{2}}\|u_{n}\|_{p}^{p},$$
we can infer that
$$\frac{\partial}{\partial\xi_{n}}I_{\eta}(\xi_{n}\star u_{n})=sP_{\eta}(\xi_{n}\star u_{n}).$$
Thus, the conclusion of Lemma \ref{Lemma4} is a consequence of the following limit
$$\lim\limits_{n\rightarrow\infty}\frac{\partial}{\partial\xi_{n}}I_{\eta}(\xi_{n}\star u_{n})=0,$$
since $I^{'}_{\eta}(\xi_{n}\star u_{n})\rightarrow 0,~\text{as} ~n\rightarrow\infty.$ This completes the proof of Lemma \ref{Lemma4}.

\end{proof}
 \begin{lemma}\label{Lemma5}
Let $\{u_{n}\}\subset S_{r}(m)$ be a Palais-Smale sequence for  $I_{\eta}$ with the level $E_{\eta}(m)$. If $\lim\limits_{n\rightarrow\infty}P_{\eta}(u_{n})=0,$
then  $\{u_{n}\}$ is bounded in $S_{r}(m)$.
\end{lemma}

\begin{proof}
We note that $\zeta_{p}p>2,$ since $p>2+4s/N$. It
follows
from
 $\lim\limits_{n\rightarrow\infty}P_{\eta}(u_{n})=0$
 that
$$I_{\eta}(u_{n})=\frac{\eta}{2p}(\zeta_{p}p-2)\|u_{n}\|_{p}^{p}+\frac{s}{N}\|u_{n}\|_{2_{s}^{*}}^{2_{s}^{*}}+o(1),$$
and from the boundedness of $I_{\eta}(u_{n}),$ that $\{\|u_{n}\|_{p}^{p}\}$ and $\{\|u_{n}\|_{2_{s}^{*}}^{2_{s}^{*}}\}$ are both bounded, therefore $\{[u_{n}]^{2}_{H^{s}}\}$ is bounded.
 This completes the proof of Lemma \ref{Lemma5}.
\end{proof}

\section{Compactness condition}\label{sec3}%2
In this section we give a very important proof of the compactness conditions,  inspired by the ideas of Soave \cite{NS}.
\begin{proposition}\label{pro}
Let $\{u_{n}\}\subset S_{r}(m)$ be a Palais-Smale sequence for  $I_{\eta}$ with the level $$0<E_{\eta}(m)<\frac{s\mathcal{S}^{\frac{N}{2s}}}{N},$$ where $\mathcal{S}$ is the
 best fractional Sobolev constant defined
 in
  \eqref{sobolev}. If $\lim\limits_{n\rightarrow\infty}P_{\eta}(u_{n})=0,$ then one of the following properties holds:

$(1)$ either up to subsequence, $u_{n}\rightharpoonup \widetilde{u}$ converges weakly in $H^{s}(\mathbb{R}^{N})$ but not strongly, where $\widetilde{u}\not\equiv0$
is a solution of the first equation  of \eqref{p} for some $\mu<0,$
and
 $$I_{\eta}(\widetilde{u})<E_{\eta}(m)-\frac{s\mathcal{S}^{\frac{N}{2s}}}{N};$$

$(2)$ or up to subsequence, $u_{n}\rightarrow \widetilde{u}$ converges strongly in $H^{s}(\mathbb{R}^{N})$, $I_{\eta}(\widetilde{u})=E_{\eta}(m)$, and $\widetilde{u}$ is a solution of \eqref{p} for some $\mu<0$.

\end{proposition}

\begin{proof}
In general, the embedding $H^{s}(\mathbb{R}^{N})\hookrightarrow L^{p}(\mathbb{R}^{N})$ is not compact for any $p\in(2,2_{s}^{*})$, so we need to restore compactness in the radial function space. According to Lemma \ref{Lemma5}, we know that sequence $\{u_{u}\}$ is  bounded and  the embedding $H^{s}_{rad}(\mathbb{R}^{N})\hookrightarrow L^{p}(\mathbb{R}^{N})$ is compact for every $p\in(2,2_{s}^{*})$ (see Lions \cite[Proposition I.1]{pll}). Therefore there exists $\widetilde{u}\in H^{s}_{rad}(\mathbb{R}^{N})$ such that up to a subsequence, $u_{n}\rightharpoonup \widetilde{u}$ converges
 weakly in $H^{s}(\mathbb{R}^{N})$, $u_{n}\rightarrow \widetilde{u}$ converges strongly in $L^{p}(\mathbb{R}^{N})$, and a.e. in $\mathbb{R}^{N}$. Since $\{u_{n}\}$ is a Palais-Smale sequence for  $I_{\eta}\mid_{S(m)},$ by the Lagrange multipliers rule there exists $\{\mu_{n}\}\subset \mathbb{R}$ such that for every $\phi\in H^{s}(\mathbb{R}^{N}),$
\begin{align}\label{3.1}
\int\int_{\mathbb{R}^{2N}}
\frac{(u_{n}(x)-u_{n}(y))(\phi(x)-\phi(y))}{|x-y|^{N+2s}}dxdy-\int_{\mathbb{R}^{N}}\Big(\mu_{n}u_{n}\phi+|u_{n}|^{2^{*}_{s}-2}u_{n}\phi+\eta|u_{n}|^{p-2}u_{n}\phi \Big)dx
=o(1)\|\phi\|,
\end{align}
as $n\rightarrow\infty$. Setting $\phi=u_{n}$, we
can
 infer that $\{\mu_{n}\}$ is also bounded, and therefore up to a subsequence, $\mu_{n}\rightarrow\mu\in \mathbb{R}^{N}$. By invoking
 $\lim\limits_{n\rightarrow\infty}P_{\eta}(u_{n})=0,$ the
 compactness of the embedding $H^{s}_{rad}(\mathbb{R}^{N})\hookrightarrow L^{p}(\mathbb{R}^{N}),$ and $\zeta_{p}<1$, we
 get
\begin{align}\label{3.2}
\mu m^{2}=&\lim\limits_{n\rightarrow\infty} \mu_{n}\|u_{n}\|_{2}^{2}=\lim\limits_{n\rightarrow\infty}\left( [u_{n}]^{2}_{H^{s}}-\|u_{n}\|_{2_{s}^{*}}^{2_{s}^{*}}-\eta\|u_{n}\|_{p}^{p}\right)\nonumber\\
=&\lim\limits_{n\rightarrow\infty}\eta(\zeta_{p}-1)\|u_{n}\|_{p}^{p}=\eta(\zeta_{p}-1)\|\widetilde{u}\|_{p}^{p}\leq0,
\end{align}
where
 $\mu=0$ if and only if $\widetilde{u}\equiv0$.

We shall now show that
\begin{align}\label{claim}
  \widetilde{u}\not\equiv0.
\end{align}
Suppose to the contrary, that $\widetilde{u}\equiv0$. Since $\{u_{n}\}$ is bounded in $H^{s}(\mathbb{R}^{N}),$
it follows that up to a subsequence,
$[u_{n}]^{2}_{H^{s}}\rightarrow\gamma\in \mathbb{R}$.
Since $P_{\eta}(u_{n})\rightarrow0$ and $u_{n}\rightarrow0$ converges strongly in $L^{p}(\mathbb{R}^{N})$, it follows that $$\|u_{n}\|_{2_{s}^{*}}^{2_{s}^{*}}=[u_{n}]^{2}_{H^{s}}-\eta\zeta_{p}\|u_{n}\|_{p}^{p}\rightarrow\gamma,$$ therefore by \eqref{sobolev},
 $\gamma\geq\mathcal{S}\gamma^{\frac{2}{2_{s}^{*}}}.$
Furthermore, we can infer that
$$~\text{either}~ \gamma=0 ~\text{or}~ \gamma>\mathcal{S}^{\frac{N}{2s}}.$$
If $\gamma>\mathcal{S}^{\frac{N}{2s}},$
then due to $I_{\eta}(u_{n})\rightarrow E_{\eta}(m)$ and  $\lim\limits_{n\rightarrow\infty}P_{\eta}(u_{n})=0,$ we get
\begin{align*}
E_{\eta}(m)+o(1)=&I_{\eta}(u_{n})=\frac{s}{N}[u_{n}]^{2}_{H^{s}}-\frac{\eta}{p}\left(1-\frac{\zeta_{p}p}{2_{s}^{*}}\right)\|u_{n}\|_{p}^{p}+o(1)\\
=&\frac{s}{N}[u_{n}]^{2}_{H^{s}}+o(1)=\frac{\gamma s}{N}+o(1),
\end{align*}
so $E_{\eta}(m)=\frac{\gamma s}{N}$, thereby  $E_{\eta}(m)\geq\frac{s\mathcal{S}^{\frac{N}{2s}}}{N}$,  which contradicts our conditions.

If instead, we have $\gamma=0$, we note that $[u_{n}]^{2}_{H^{s}}\rightarrow0$, $\|u_{n}\|_{2_{s}^{*}}^{2_{s}^{*}}\rightarrow0$ and $\|u_{n}\|_{p}^{p}\rightarrow0$. Therefore $I_{\eta}(u_{n})\rightarrow0$, which is a contradiction as well. So \eqref{claim} is proved.
Furthermore, it follows from \eqref{3.2} and \eqref{claim} that $\mu<0$. Invoking
 the limit weak convergence in \eqref{3.1}, we get
\begin{align}\label{3.4}
(-\Delta)^{s} \widetilde{u}=\mu \widetilde{u}+|\widetilde{u}|^{2^{*}_{s}-2}\widetilde{u}+\eta|\widetilde{u}|^{p-2}\widetilde{u}~~ ~\text{in}~\mathbb{R}^{N},
\end{align}
and thus by the Pohozaev identity (see Chang and Wang \cite[Proposition 4.1]{XC}) and related explanations in Zhen and Zhang \cite[Proposition 2.1 and Remark 2.1]{MZ}, we have $P_{\eta}(\widetilde{u})=0$. We know that $w_{n}=u_{n}-\widetilde{u}\rightharpoonup0$ in $H^{s}(\mathbb{R}^{N})$, and  according to Zuo et al.
\cite[Lemma 2.4]{zuo} and  the Br\'{e}zis-Lieb lemma \cite{bl}, we have

\begin{align}\label{3.5}
\begin{split}
[u_{n}]^{2}_{H^{s}}&=[\widetilde{u}]^{2}_{H^{s}}+[w_{n}]^{2}_{H^{s}}+o(1),\\ ~ \|u_{n}\|_{2_{s}^{*}}^{2_{s}^{*}}&=\|\widetilde{u}\|_{2_{s}^{*}}^{2_{s}^{*}}+\|w_{n}\|_{2_{s}^{*}}^{2_{s}^{*}}+o(1).
\end{split}
\end{align}

Thus, by  $\lim\limits_{n\rightarrow\infty}P_{\eta}(u_{n})=0$ and since $u_{n}\rightarrow \widetilde{u}$
converges
strongly in $L^{p}$, we obtain $$[\widetilde{u}]^{2}_{H^{s}}+[w_{n}]^{2}_{H^{s}}=\eta\zeta_{p}\|\widetilde{u}\|_{p}^{p}+\|\widetilde{u}\|_{2_{s}^{*}}^{2_{s}^{*}}+\|w_{n}\|_{2_{s}^{*}}^{2_{s}^{*}}+o(1).$$
In view of $P_{\eta}(\widetilde{u})=0$, we also have that $$[w_{n}]^{2}_{H^{s}}=\|w_{n}\|_{2_{s}^{*}}^{2_{s}^{*}}+o(1).$$ We claim that up to a subsequence
$$\lim\limits_{n\rightarrow\infty}[w_{n}]^{2}_{H^{s}}=\lim\limits_{n\rightarrow\infty}\|w_{n}\|_{2_{s}^{*}}^{2_{s}^{*}}=\gamma\geq0,~~ \Rightarrow~~ \gamma\geq\mathcal{S}\gamma^{\frac{2}{2_{s}^{*}}}$$
thanks to \eqref{sobolev}. Hence, either $\gamma=0$ or $\gamma>\mathcal{S}^{\frac{N}{2s}}$.

If $\gamma>\mathcal{S}^{\frac{N}{2s}}$, then from \eqref{3.5}, we obtain that
$$E_{\eta}(m)=\lim\limits_{n\rightarrow\infty}I_{\eta}(u_{n})=\lim\limits_{n\rightarrow\infty}\left(I_{\eta}(\widetilde{u})+\frac{1}{2}[w_{n}]^{2}_{H^{s}}
-\frac{1}{2_{s}^{*}}\|w_{n}\|_{2_{s}^{*}}^{2_{s}^{*}}\right)=I_{\eta}(\widetilde{u})+\frac{s\gamma}{N}\geq I_{\eta}(\widetilde{u})+\frac{s\mathcal{S}^{\frac{N}{2s}}}{N},$$
whence alternative (1) in the assertion of the proposition follows, i.e., up to a subsequence $u_{n}\rightharpoonup \widetilde{u}$ converges weakly in $H^{s}(\mathbb{R}^{N})$ but not strongly, where $\widetilde{u}\not\equiv0$
is a solution of the first equation  of \eqref{p} for some $\mu<0,$
and
 $$I_{\eta}(\widetilde{u})<E_{\eta}(m)-\frac{s\mathcal{S}^{\frac{N}{2s}}}{N}.$$

If instead, we have $\gamma=0$, then we claim that $u_{n}\rightarrow \widetilde{u}$ in $H^{s}(\mathbb{R}^{N})$. Indeed, we have $\lim\limits_{n\rightarrow\infty}[w_{n}]^{2}_{H^{s}}=0$, so it follows from  $w_{n}=u_{n}-\widetilde{u}$ that $[u_{n}-\widetilde{u}]^{2}_{H^{s}}\rightarrow0$.

Next, it suffices to verify that $u_{n}\rightarrow \widetilde{u}$ in $L^{2}$. Choosing $\phi=u_{n}-\widetilde{u}$ in \eqref{3.1}, invoking \eqref{3.4} with $u_{n}\rightarrow \widetilde{u}$ and subtracting, we get
\begin{align*}
[u_{n}-&\widetilde{u}]^{2}_{H^{s}}-\int_{\mathbb{R}^{N}}\left(\mu_{n}u_{n}-\mu\widetilde{u}\right)\left(u_{n}-\widetilde{u}\right)dx=\\
&\int_{\mathbb{R}^{N}}\left(|u_{n}|^{2_{s}^{*}-2}u_{n}-|\widetilde{u}|^{2_{s}^{*}-2}\widetilde{u}\right)\left(u_{n}-\widetilde{u}\right)dx+\int_{\mathbb{R}^{N}}\left(|u_{n}|^{p-2}u_{n}-|\widetilde{u}|^{p-2}\widetilde{u}\right)\left(u_{n}-\widetilde{u}\right)dx+o(1).
\end{align*}
 We note that $\lim\limits_{n\rightarrow\infty}\|w_{n}\|_{2_{s}^{*}}^{2_{s}^{*}}=0.$
 It follows from \eqref{3.5} that $\|u_{n}\|_{2_{s}^{*}}^{2_{s}^{*}}\rightarrow\|\widetilde{u}\|_{2_{s}^{*}}^{2_{s}^{*}}$, therefore, in the formula above, the first term, the third term, and the fourth term converge to 0. As a result,  $$0=\lim\limits_{n\rightarrow\infty}\int_{\mathbb{R}^{N}}\left(\mu_{n}u_{n}-\mu\widetilde{u}\right)\left(u_{n}-\widetilde{u}\right)dx=\lim\limits_{n\rightarrow\infty}\mu\int_{\mathbb{R}^{N}}\left(u_{n}-\widetilde{u}\right)^{2}dx.$$
Thus also assertion (2) of Proposition \ref{pro} has been established, i.e., up to subsequence, $u_{n}\rightarrow \widetilde{u}$ converges strongly in $H^{s}(\mathbb{R}^{N})$, $I_{\eta}(\widetilde{u})=E_{\eta}(m)$, and $\widetilde{u}$ is a solution of \eqref{p} for some $\mu<0$. The proof of   Proposition \ref{pro} is complete.
\end{proof}
\section{Proof of Theorem \ref{T1.1}}\label{sec4}
Lemma \ref{Lemma4} and \cite[Propositions 5.3-5.4]{HZ} imply that for a
given Palais-Smale sequence $\{u_{n}\}\subset S_{r}(m)$ for  $I_{\eta}$ with the level $E_{\eta}(m)$, if  $\lim\limits_{n\rightarrow\infty}P_{\eta}(\xi_{n}\star u_{n})=0,$ then the sequence $\{\xi_{n}\star u_{n}\}\subset S_{r}(m)$ is also a Palais-Smale sequence for  $I_{\eta}$ on
the same level, thus we can apply Lemma \ref{Lemma5}. In order to prove our
main result, it remains to verify the condition $E_{\eta}(m)<\frac{s\mathcal{S}^{\frac{N}{2s}}}{N}$ of Proposition \ref{pro}, which is a consequence of Lemma \ref{Lemma3}.

  Therefore, we know that one of the two conclusions of Proposition \ref{pro} must hold.  We show that conclusion (1) does not hold. Indeed, if it did,
  then
   $\widetilde{u}$
   would be
    a nontrivial solution of \eqref{p}, i.e., up to a subsequence $u_{n}\rightharpoonup \widetilde{u}$ converges weakly in $H^{s}(\mathbb{R}^{N})$ but not strongly, where $\widetilde{u}\not\equiv0$
is a solution of the first equation  of \eqref{p} for some $\mu<0,$
and
 $$I_{\eta}(\widetilde{u})<E_{\eta}(m)-\frac{s\mathcal{S}^{\frac{N}{2s}}}{N}<0.$$
However, since $P_{\eta}(\widetilde{u})=0$  by the Pohozaev identity  and $\zeta_{p}p>2,$ we also get
 $$I_{\eta}(\widetilde{u})=\frac{\eta}{2p}(\zeta_{p}p-2)\|u_{n}\|_{p}^{p}+\frac{s}{N}\|u_{n}\|_{2_{s}^{*}}^{2_{s}^{*}}>0,$$
 which is a contradiction.

 Therefore, conclusion (2)
 must hold and $\widetilde{u}$ is a radial normalized solution of \eqref{p} for some $\mu<0$. This completes the proof of Theorem \ref{T1.1}.

\vfill\eject

\section*{Acknowledgements}

Jiabin Zuo was supported by the Guangdong Basic and Applied Basic Research
Foundation (2022A1515110907) and the Project funded by China
Postdoctoral Science Foundation (2023M730767). Yuyou Zhong was supported by Innovative Research Funding Program for Graduate Students of Guangzhou University (2022GDJC-D09).
Du\v{s}an D. Repov\v{s} was supported by the Slovenian Research Agency grants P1-0292, N1-0278, N1-0114, N1-0083, J1-4031, and J1-4001.
 We thank the referee for many important comments and suggestions, which have significantly improved the presentation.

\end{document}